\documentclass[11pt]{article}

\usepackage{amsmath,amssymb,amsthm}
\usepackage{enumitem}
\usepackage{booktabs}
\usepackage{graphicx}
\usepackage{flafter}
\usepackage{hyperref}
\hypersetup{
  hidelinks,
  pdftitle={G-Exponential Families through Probability Coordinates},
  pdfauthor={Manuela-Simona Cojocea}
}

\newtheorem{theorem}{Theorem}[section]
\newtheorem{proposition}[theorem]{Proposition}

\newtheorem{corollary}[theorem]{Corollary}
\theoremstyle{definition}
\newtheorem{definition}[theorem]{Definition}

\newtheorem{remark}[theorem]{Remark}
\newtheorem{example}[theorem]{Example}

\newcommand{\E}{\mathbb{E}}
\newcommand{\Prob}{\mathbb{P}}
\newcommand{\R}{\mathbb{R}}
\newcommand{\Var}{\operatorname{Var}}
\newcommand{\Unif}{\mathrm{Unif}}
\newcommand{\KL}{\operatorname{KL}}
\newcommand{\IF}{\operatorname{IF}}

\title{$G$-Exponential Families through Probability Coordinates}
\author{Manuela-Simona Cojocea}
\date{}

\begin{document}

\maketitle

\begin{abstract}
We develop a probability-geometric construction of generalized exponential
families based on the probability coordinate charts of the companion papers:
a classical exponential family $p_\eta$ on the unit interval is transported
to value space through a chart $G$, yielding the \emph{$G$-exponential
family} $f_\eta(x)=p_\eta(G(x))\,g(x)$. The construction deforms the
geometry in which exponential structure is represented, never the
exponential function itself, in deliberate contrast with Tsallis- and
Kaniadakis-type algebraic deformations.

The transported families inherit the full structure of the companion
framework. The probability coordinate $U=G(X)$ of a family
member has law exactly $p_\eta$, so initial Kolmogorov moments are pullbacks
of classical moments of $p_\eta$; in particular, for the canonical family
the probability barycenter is the pullback of the mean parameter,
$b_G(f_\eta)=G^{-1}(A'(\eta))$, and the Fisher information coincides with
the Kolmogorov variance. Tail behaviour is a family invariant inherited
from the chart: every member is tail-equivalent to the chart density, so
heavy tails arise from geometry and are neutral under tilting in
probability coordinates---the exact opposite of classical exponential
tilting. On the statistical side, the sufficient statistic is bounded by
construction, maximum likelihood reduces to moment matching in probability
coordinates, and the maximum likelihood estimator has bounded influence
function: robustness arises from geometric compactification, in the sense
of the companion inference paper. We further characterize $G$-exponential
families as minimizers of relative entropy to the chart law under
probability-coordinate moment constraints, and show that transport leaves
the information geometry of the family invariant. A worked example based on
the Cauchy chart illustrates the theory with explicit formulas, and a Monte
Carlo study on this family---in which every member has infinite mean---confirms
the efficient, calibrated, and bounded-influence behaviour of the
transported maximum likelihood estimator.
\end{abstract}

\section{Introduction}\label{sec:intro}

Classical exponential families occupy a central position in probability and
statistics: they support finite-dimensional sufficient statistics, convex
likelihood theory, entropy-maximization characterizations, and a rich
information geometry \cite{BarndorffNielsen1978,Brown1986,Amari2016}. A
family of densities belongs to this class when it can be written in the
form
\[
f_\theta(x)=h(x)\exp\bigl(\eta(\theta)T(x)-A(\theta)\bigr),
\]
so that the logarithm of the density is affine in the sufficient statistic.
The exponential decay implicit in this representation, however, is a
value-space property, and it excludes many models of practical importance:
heavy-tailed laws, power-law phenomena, and distributions with boundary
concentration do not admit exponential representations in value space.

A substantial literature addresses this limitation by deforming the
exponential function itself. Tsallis-type $q$-exponentials
\cite{Tsallis1988}, Kaniadakis exponentials \cite{Kaniadakis2002}, and the
generalized exponential families organized around deformed logarithms
\cite{Naudts2011} replace $\exp$ by a modified function whose slower decay
accommodates power-law tails. A related line, going back to R\'enyi
\cite{Renyi1961}, deforms the aggregation of probability weights in
entropy-based constructions. In all these approaches the deformation is
\emph{algebraic}: the exponential or logarithmic operation is modified
directly.

The present paper takes a different route, dictated by the philosophy of
the companion papers \cite{Cojocea2026a,Cojocea2026b}: we deform the
\emph{geometry} in which exponential structure is represented, and never
the exponential function itself. In the framework of \cite{Cojocea2026a}, a
continuous distribution function is interpreted as a \emph{probability
coordinate chart} transporting value space into the unit interval, and
averaging performed in probability coordinates defines the
\emph{probability barycenter} $b_G(X)=G^{-1}(\E[G(X)])$, a Kolmogorov
expectation in the population-level sense of de~Carvalho
\cite{deCarvalho2016}; the construction descends from the classical theory
of generalized means \cite{Kolmogorov1930,Nagumo1930,Aczel1966}. The
companion paper \cite{Cojocea2026b} develops the associated statistical
inference and shows that robustness arises from \emph{geometric
compactification}: probability coordinates are bounded, so the estimators
built from them have bounded influence without any reweighting or
truncation.

Here the same mechanism is applied to exponential structure. We place a
classical exponential family $p_\eta$ on the probability space $(0,1)$ and
transport it to value space through a chart $G$, obtaining the
\emph{$G$-exponential family}
\[
f_\eta(x)=p_\eta(G(x))\,g(x)
        =g(x)\exp\bigl(\eta\,T(G(x))-A(\eta)\bigr),
\]
where $g=G'$ is the chart density. Exponentiality remains classical and
untouched in probability coordinates; heavy tails, boundary concentration,
and robustness properties of the transported family are produced entirely
by the geometry of the chart. The construction is an exponential tilt of
the chart law performed in probability coordinates, and because the tilting
factor is bounded, it exists for every chart---including heavy-tailed
charts, for which classical exponential tilting is unavailable.

The contributions of the paper are the following.
\begin{enumerate}[leftmargin=2em]
\item We define $G$-exponential families and characterize them by a
pushforward identity: $X\sim f_\eta$ if and only if the probability
coordinate $U=G(X)$ has law $p_\eta$
(Section~\ref{sec:gexp}). Normalization and the transported log-density
structure follow immediately, and the carrier of the family on $(0,1)$ can
always be absorbed into the chart, so that base measure and geometry are
identified.
\item We compute the Kolmogorov structure of the transported families in
the sense of \cite{Cojocea2026a}: initial Kolmogorov moments are pullbacks
of classical moments of $p_\eta$, the probability barycenter of the
canonical family is the pullback of the mean parameter,
$b_G(f_\eta)=G^{-1}(A'(\eta))$, the Kolmogorov variance equals the Fisher
information, and the moment-determinacy theorem of \cite{Cojocea2026a}
identifies each member from its Kolmogorov moment sequence
(Section~\ref{sec:kolmogorov}).
\item We prove a tail-equivalence theorem: every member of a
$G$-exponential family is tail-equivalent to the chart density, so the
regular-variation index of the tails is a family invariant
(Section~\ref{sec:tails}). Tilting in probability coordinates is
tail-neutral, in exact opposition to classical exponential tilting.
\item We develop the statistical structure in the format of
\cite{Cojocea2026b}: the sufficient statistic is bounded by construction,
maximum likelihood is moment matching in probability coordinates, and the
maximum likelihood estimator is asymptotically normal with \emph{bounded}
influence function---robustness from geometric compactification, with an
honest accounting of the anchoring caveat inherited from the fixed chart
(Section~\ref{sec:inference}).
\item We characterize $G$-exponential families variationally, as
minimizers of relative entropy to the chart law under probability-coordinate
moment constraints, and we show that the transport leaves relative entropy
and Fisher information invariant: the chart changes the geometry of value
space, not the information geometry of the family
(Section~\ref{sec:entropy}).
\end{enumerate}

The paper is organized as follows. Section~\ref{sec:framework} recalls the
probability-coordinate framework of the companion papers and fixes
notation. Section~\ref{sec:gexp} defines $G$-exponential families and
establishes their basic transport structure. Section~\ref{sec:kolmogorov}
computes their Kolmogorov moments and barycenters.
Section~\ref{sec:tails} establishes tail inheritance.
Section~\ref{sec:inference} develops estimation and robustness.
Section~\ref{sec:entropy} treats entropy and information invariance.
Section~\ref{sec:example} works out the Cauchy chart in full detail, and
Section~\ref{sec:simulation} reports a Monte Carlo study on that family.
Section~\ref{sec:deformed} compares the construction with
deformed-exponential and entropy-based frameworks.
Section~\ref{sec:intrinsic} discusses the programmatic question of
intrinsically geometric exponentiality, and
Section~\ref{sec:discussion} concludes.

\section{The probability-coordinate framework}\label{sec:framework}

We briefly recall the framework of the companion papers; definitions are
aligned with \cite{Cojocea2026a}, to which we refer for the geometric
development, and are restated here only to fix notation.

\begin{definition}[Probability coordinate chart \cite{Cojocea2026a}]
\label{def:chart}
Let $\widetilde G:\R\to[0,1]$ be a continuous cumulative distribution
function and suppose there exists an interval $I=(a,b)\subseteq\R$, with
$-\infty\le a<b\le+\infty$, such that the restriction
$G=\widetilde G|_I$ is strictly increasing and satisfies
$\lim_{x\to a^+}G(x)=0$ and $\lim_{x\to b^-}G(x)=1$. The resulting
bijection $G:I\to(0,1)$ is called the \emph{probability coordinate chart}
induced by $\widetilde G$.
\end{definition}

Throughout the paper, charts are assumed absolutely continuous with
density $g=G'>0$ on $I$; we refer to the law with distribution function
$G$ as the \emph{chart law}. For an $I$-valued random variable $X$, the
\emph{probability coordinate} of $X$ is $U=G(X)\in(0,1)$, and the
\emph{probability barycenter} (or Kolmogorov expectation induced by $G$)
is
\[
b_G(X)=G^{-1}\bigl(\E[G(X)]\bigr),
\]
which exists for arbitrary $X$ because $G(X)$ is bounded
\cite{Cojocea2026a}. For $r\ge1$, the \emph{initial Kolmogorov moment} of
order $r$ is
\begin{equation}\label{eq:initial-moment}
M_r^{(G)}(X)=G^{-1}\bigl(\E[(G(X))^r]\bigr),
\end{equation}
and, with $\bar u=\E[G(X)]$, the \emph{centred Kolmogorov moments}
$\E[(G(X)-\bar u)^r]$ are reported on the probability scale and are
deliberately not pulled back through $G^{-1}$; the \emph{Kolmogorov
variance} is $V_G(X)=\Var(G(X))\in[0,\tfrac14]$
\cite{Cojocea2026a}. When the chart is intrinsic, $G=F_X$ with $X$
continuous, the probability integral transform gives
$F_X(X)\sim\Unif(0,1)$, the barycenter is the median, and
$\E[(F_X(X))^r]=1/(r+1)$.

The statistical theory of these functionals---consistency and asymptotic
normality of plug-in estimators, the degeneracy of the fully empirical
intrinsic barycenter, estimated charts, and the robustness analysis through
influence functions---is developed in \cite{Cojocea2026b}. We will use its
formats and its honest-accounting conventions when we discuss inference in
Section~\ref{sec:inference}.

\section{\texorpdfstring{$G$}{G}-exponential families}\label{sec:gexp}

\subsection{Exponential structure in probability coordinates}

Classical exponential families impose exponential structure in value
space. The present construction imposes it in probability coordinates
instead.

\begin{definition}[Exponential family on probability space]
\label{def:expfam01}
Let $T:(0,1)\to\R$ be measurable and non-constant, and let
\[
N=\Bigl\{\eta\in\R:\int_0^1 e^{\eta T(u)}\,du<\infty\Bigr\}
\]
be the natural parameter set. For $\eta\in N$, define the density on
$(0,1)$
\begin{equation}\label{eq:p-eta}
p_\eta(u)=\exp\bigl(\eta\,T(u)-A(\eta)\bigr),
\qquad
A(\eta)=\log\int_0^1 e^{\eta T(u)}\,du .
\end{equation}
\end{definition}

We use the natural parameter throughout. A curved parametrization
$\eta=\eta(\theta)$ can be composed afterwards and changes nothing
structural. On the bounded domain $(0,1)$ the natural parameter
set is particularly generous: if $T$ is bounded---for instance, if $T$
extends continuously to $[0,1]$---then $N=\R$.

\begin{definition}[$G$-exponential family]\label{def:gexp}
Let $G$ be a probability coordinate chart on $I$ with density $g$, and let
$\{p_\eta:\eta\in N\}$ be an exponential family on $(0,1)$ as in
Definition~\ref{def:expfam01}. The \emph{$G$-exponential family} generated
by $(G,T)$ is the family of densities on $I$
\begin{equation}\label{eq:gexp}
f_\eta(x)=p_\eta(G(x))\,g(x)
        =g(x)\exp\bigl(\eta\,T(G(x))-A(\eta)\bigr),
\qquad \eta\in N .
\end{equation}
\end{definition}

The construction transports exponential structure through the chart:
exponentiality is classical and unchanged on the probability scale, while
the geometry of the resulting family on value space is entirely determined
by $G$. The deformation is geometric, not algebraic.

\subsection{The pushforward characterization}

The defining property of the transported family is most transparently
expressed through the probability coordinate itself.

\begin{proposition}[Pushforward characterization]\label{prop:pushforward}
Let $G$ be a chart on $I$ with density $g$, and let $X$ be an $I$-valued
random variable. Then $X$ has density $f_\eta$ of \eqref{eq:gexp} if and
only if the probability coordinate $U=G(X)$ has density $p_\eta$ on
$(0,1)$. In particular, each $f_\eta$ is a probability density on $I$.
\end{proposition}

\begin{proof}
Suppose $X$ has density $f_\eta$. For $u\in(0,1)$, since $G:I\to(0,1)$ is
an increasing bijection,
\[
\Prob(U\le u)
=\Prob\bigl(X\le G^{-1}(u)\bigr)
=\int_a^{G^{-1}(u)}p_\eta(G(t))\,g(t)\,dt
=\int_0^u p_\eta(v)\,dv,
\]
by the substitution $v=G(t)$. Hence $U$ has density $p_\eta$. Conversely,
if $U=G(X)$ has density $p_\eta$, then $X=G^{-1}(U)$ and, for $x\in I$,
$\Prob(X\le x)=\int_0^{G(x)}p_\eta(v)\,dv$, whose derivative in $x$ is
$p_\eta(G(x))g(x)=f_\eta(x)$. Taking $u\to1^-$ in the first display gives
$\int_I f_\eta=\int_0^1 p_\eta=1$, so $f_\eta$ is a probability density.
\end{proof}

\begin{remark}[Transport of the log-density]\label{rem:transport-log}
Taking logarithms in \eqref{eq:gexp},
\[
\log f_\eta(x)=\eta\,T(G(x))-A(\eta)+\log g(x),
\]
so on value space the family is exponential \emph{in the transported
statistic} $T\circ G$, with carrier $g$. All exponential-family structure
(convexity of $A$, sufficiency, likelihood theory) is preserved; what
changes is the geometry through which the statistic reads the data, since
$T\circ G$ measures probabilistic location rather than magnitude.
\end{remark}

\begin{remark}[Exponential tilting in probability coordinates]
\label{rem:tilting}
If $0\in N$---which always holds when $T$ is bounded---then $f_0=g$: the
chart law itself is a member of the family, and \eqref{eq:gexp} exhibits
every other member as an exponential tilt of the chart law by the
\emph{bounded} factor $\exp(\eta T(G(x))-A(\eta))$. Two consequences
deserve emphasis. First, the tilt exists for every chart, including
heavy-tailed charts for which the classical tilt $e^{\eta x}$ is not
integrable for any $\eta\neq0$. Second, the tilting acts on the geometry
of representation and not on the underlying measure-theoretic weighting of
tail events, in contrast with escort transformations
\cite{Tsallis1988}, which reweight the density nonlinearly; the comparison
of geometric deformation with measure deformation is developed in
\cite{Cojocea2026a}.
\end{remark}

\begin{remark}[The carrier can be absorbed into the chart]
\label{rem:carrier}
Definition~\ref{def:expfam01} takes Lebesgue measure as base measure on
$(0,1)$. A family with carrier, $p_\eta(u)=h(u)\exp(\eta T(u)-A(\eta))$
with $h>0$ and $0<\int_0^1 h<\infty$, reduces to the carrier-free case by
a change of chart: with $H(u)=\int_0^u h/\int_0^1 h$, the composition
$G_h=H\circ G$ is again a probability coordinate chart on $I$, and the
transported family generated by $(G,T)$ with carrier $h$ coincides, up to
the reparametrization $T\mapsto T\circ H^{-1}$, with the carrier-free
$G_h$-exponential family. In probability geometry, base measure and
coordinate geometry are therefore interchangeable, and we take $h\equiv1$
without loss of generality.
\end{remark}

\begin{remark}[Classical exponential families as a boundary case]
\label{rem:classical}
The classical natural exponential family generated by the chart law has
density
\[
f_\eta(x)=g(x)e^{\eta x-A(\eta)}.
\]
It is formally the $G$-exponential family with statistic $T=G^{-1}$, the
inverse chart itself. This statistic is
unbounded precisely when $I$ is unbounded, and then $N$ may collapse: for
heavy-tailed charts, $N=\{0\}$. The boundedness of $T$ is thus the exact
frontier between the classical regime, where tilting modifies tails and
may fail to exist, and the probability-coordinate regime, where tilting is
universal and, as shown in Section~\ref{sec:tails}, tail-neutral.
\end{remark}

\subsection{The canonical family}

The simplest choice of statistic is the identity in probability
coordinates.

\begin{example}[Canonical $G$-exponential family]\label{ex:canonical}
Take $T(u)=u$. Then $N=\R$ and, for $\eta\neq0$,
\begin{equation}\label{eq:canonical-A}
A(\eta)=\log\frac{e^\eta-1}{\eta},
\qquad
A'(\eta)=\frac{1}{1-e^{-\eta}}-\frac1\eta,
\qquad
A''(\eta)=\frac{1}{\eta^2}-\frac{e^\eta}{(e^\eta-1)^2},
\end{equation}
with $A(0)=0$, $A'(0)=\tfrac12$, $A''(0)=\tfrac1{12}$ by continuity. The
probability-space family $p_\eta(u)=\eta e^{\eta u}/(e^\eta-1)$ is the
truncated exponential family on $(0,1)$, and the transported family
\[
f_\eta(x)=\frac{\eta\,e^{\eta G(x)}}{e^\eta-1}\,g(x)
\]
tilts the chart law towards its upper tail for $\eta>0$ and towards its
lower tail for $\eta<0$. The value $A''(0)=\tfrac1{12}$ is the variance of
the uniform law, that is, the Kolmogorov variance of the chart law under
its own chart, consistent with the intrinsic computations of
\cite{Cojocea2026a}.
\end{example}

\section{Kolmogorov structure of the transported families}
\label{sec:kolmogorov}

Here we compute the framework's own functionals---barycenters and
Kolmogorov moments---for the new families. The pushforward characterization
makes the calculation exact \cite{Cojocea2026a,Cojocea2026b}.

\begin{proposition}[Kolmogorov moments are pulled-back moments of
$p_\eta$]\label{prop:moments}
Let $X\sim f_\eta$ as in \eqref{eq:gexp}. Then for every $r\ge1$,
\[
\begin{aligned}
\E[(G(X))^r]
  &=\int_0^1 u^r\,p_\eta(u)\,du,\\
M_r^{(G)}(X)
  &=G^{-1}\Bigl(\int_0^1 u^r p_\eta(u)\,du\Bigr).
\end{aligned}
\]
All initial Kolmogorov moments exist, the centred Kolmogorov moments of
$X$ coincide with the centred moments of $p_\eta$ on the probability
scale, and $V_G(X)=\Var_{p_\eta}(U)$. Moreover, by the moment-determinacy
theorem of \cite{Cojocea2026a}, the sequence
$\{M_r^{(G)}(X)\}_{r\ge1}$ determines the member $f_\eta$ uniquely.
\end{proposition}

\begin{proof}
By Proposition~\ref{prop:pushforward}, $U=G(X)$ has density $p_\eta$, so
all probability-space moments of $X$ under the chart $G$ are classical
moments of $p_\eta$; existence follows from boundedness of $U$. The
pullback expression is \eqref{eq:initial-moment}. Determinacy is the
moment-determinacy theorem of \cite{Cojocea2026a} applied to the chart $G$
and the member law.
\end{proof}

The most useful specialization connects the barycenter of the family to
the convex duality of exponential families.

\begin{theorem}[Barycenters and moments as pulled-back mean parameters]
\label{thm:duality}
Fix $r\ge1$ and consider the $G$-exponential family with statistic
$T(u)=u^r$ and log-partition function $A_r$. Let $X_\eta\sim f_\eta$.
Then:
\begin{enumerate}[label=\textup{(\roman*)}]
\item $\E[(G(X_\eta))^r]=A_r'(\eta)$, and hence
\[
M_r^{(G)}(X_\eta)=G^{-1}\bigl(A_r'(\eta)\bigr);
\]
\item the map $\eta\mapsto M_r^{(G)}(X_\eta)$ is strictly increasing on
$N=\R$; in particular the $r$-th initial Kolmogorov moment identifies the
member within the family;
\item for $r=1$ (the canonical family of Example~\ref{ex:canonical}),
\[
b_G(X_\eta)=G^{-1}\bigl(A'(\eta)\bigr),
\qquad
V_G(X_\eta)=A''(\eta).
\]
\end{enumerate}
\end{theorem}

\begin{proof}
(i) For a natural exponential family, differentiation under the integral
sign in
\[
A_r(\eta)=\log\int_0^1 e^{\eta u^r}\,du
\]
gives
$A_r'(\eta)=\E_{p_\eta}[U^r]$ \cite{Brown1986}; the pullback expression is
Proposition~\ref{prop:moments}.
(ii) $A_r''(\eta)=\Var_{p_\eta}(U^r)>0$ because $U^r$ is non-degenerate
under every $p_\eta$; thus $A_r'$ is strictly increasing, and so is
$G^{-1}\circ A_r'$ since $G^{-1}$ is strictly increasing.
(iii) is (i) with $r=1$ together with
$V_G(X_\eta)=\Var_{p_\eta}(U)=A''(\eta)$.
\end{proof}

\begin{remark}[Duality made geometric]
In classical exponential-family theory, the mean parameter $A'(\eta)$ is
the coordinate dual to the natural parameter $\eta$
\cite{BarndorffNielsen1978,Brown1986,Amari2016}. Theorem~\ref{thm:duality}
shows that in the transported family this dual coordinate is precisely the
probability-scale barycenter of \cite{Cojocea2026a}, and its pullback
$b_G(X_\eta)=G^{-1}(A'(\eta))$ is the probability barycenter of the
member. The Kolmogorov dispersion interval of \cite{Cojocea2026a},
$D_G(X_\eta)=[G^{-1}(A'(\eta)-s),\,G^{-1}(A'(\eta)+s)]$ with
$s=\sqrt{A''(\eta)}$, transports one standard deviation of the mean
parameter to value space. The barycentric machinery of the companion
papers is thus not merely compatible with $G$-exponential families: it
computes their canonical parameters.
\end{remark}

\section{Tail geometry: inheritance from the chart}\label{sec:tails}

One of the central claims of the framework is that heavy tails emerge from
the geometry of the probability chart rather than from any modification of
exponential structure. The following theorem makes the claim precise, and
its proof requires nothing beyond a comparison of survival functions; in
particular, no formal differentiation of asymptotic relations is needed.

Throughout this section the upper endpoint of $I$ is $b=+\infty$; the
lower endpoint is treated symmetrically.

\begin{theorem}[Tail equivalence with the chart]\label{thm:tail-equivalence}
Let $f_\eta$ be a member of a $G$-exponential family, and suppose that the
probability-space density has a finite positive boundary limit,
\[
c_\eta:=\lim_{u\to1^-}p_\eta(u)\in(0,\infty),
\]
which holds in particular whenever $T$ extends continuously to $u=1$.
Denote by $F_\eta$ the distribution function of $f_\eta$. Then, as
$x\to\infty$,
\[
\frac{f_\eta(x)}{g(x)}\longrightarrow c_\eta
\qquad\text{and}\qquad
\frac{1-F_\eta(x)}{1-G(x)}\longrightarrow c_\eta .
\]
Every member of the family is therefore tail-equivalent to the chart law.
\end{theorem}

\begin{proof}
The first limit is immediate: $f_\eta(x)/g(x)=p_\eta(G(x))\to c_\eta$
since $G(x)\to1^-$. For the second, fix $\varepsilon\in(0,c_\eta)$ and
choose $\delta>0$ with $|p_\eta(u)-c_\eta|<\varepsilon$ for all
$u\in(1-\delta,1)$. For $x$ large enough that $G(x)>1-\delta$,
\[
1-F_\eta(x)=\int_x^\infty p_\eta(G(t))\,g(t)\,dt,
\]
and on the domain of integration $G(t)\ge G(x)>1-\delta$, so
\[
(c_\eta-\varepsilon)\int_x^\infty g(t)\,dt
\;\le\;
1-F_\eta(x)
\;\le\;
(c_\eta+\varepsilon)\int_x^\infty g(t)\,dt .
\]
Since $\int_x^\infty g=1-G(x)$, dividing by $1-G(x)$ and letting
$x\to\infty$, then $\varepsilon\to0$, gives the claim.
\end{proof}

\begin{corollary}[The tail index is a family invariant]\label{cor:rv}
If, in addition, $1-G$ is regularly varying at infinity with index
$-\beta<0$, then $1-F_\eta$ is regularly varying with the same index
$-\beta$, for every $\eta\in N$ satisfying the hypothesis of
Theorem~\ref{thm:tail-equivalence}. In particular, all members of the
family share the tail index of the chart.
\end{corollary}

\begin{proof}
Tail equivalence with a positive constant preserves regular variation and
its index \cite{BinghamGoldieTeugels1987}.
\end{proof}

\begin{remark}[Tilting is tail-neutral in probability coordinates]
\label{rem:tail-neutral}
Corollary~\ref{cor:rv} is the exact opposite of the classical situation.
Classical exponential tilting modifies tail decay whenever it exists---and
for heavy-tailed laws it does not exist at all, since $e^{\eta x}$ is not
integrable against a regularly varying density for $\eta>0$. Tilting in
probability coordinates exists universally and changes only the tail
\emph{constant} $c_\eta$, never the tail \emph{index}. Heavy tails are
thereby exhibited as an attribute of the chart geometry, shared by the
entire family and unreachable by the algebra of the tilt: geometry decides
the decay, the exponential structure decides only how probability mass is
distributed relative to that geometry. This is the boundary-concentration
picture of heavy tails of \cite{Cojocea2026a}, now expressed at the level
of parametric families.
\end{remark}

\section{Statistical structure}\label{sec:inference}

We now develop estimation for $G$-exponential families in the format of
the companion inference paper \cite{Cojocea2026b}. Throughout this
section, the chart $G$ is fixed and known---a benchmark geometry in the
sense of \cite{Cojocea2026a,Cojocea2026b}---and $T$ is bounded, with
$\tau_-=\operatorname{ess\,inf}T$ and $\tau_+=\operatorname{ess\,sup}T$,
so that $N=\R$.

Let $X_1,\dots,X_n$ be independent with density $f_\eta$, and set
$U_i=G(X_i)$, so that $U_1,\dots,U_n$ are independent with density
$p_\eta$ by Proposition~\ref{prop:pushforward}.

\subsection{Sufficiency and maximum likelihood}

From \eqref{eq:gexp}, the log-likelihood is
\begin{equation}\label{eq:loglik}
\ell_n(\eta)
=\sum_{i=1}^n\log g(X_i)
+\eta\sum_{i=1}^n T(U_i)-nA(\eta),
\end{equation}
so by the factorization criterion $S_n=\sum_{i=1}^n T(G(X_i))$ is
sufficient for $\eta$. The summands $T(G(X_i))$ are \emph{bounded} random
variables, whatever the tail behaviour of the observations: sufficiency is
carried by a statistic that lives on the probability scale.

\begin{proposition}[Maximum likelihood as moment matching in probability
coordinates]\label{prop:mle}
The log-likelihood \eqref{eq:loglik} is strictly concave in $\eta$, and
the maximum likelihood estimator $\hat\eta_n$, when it exists, is the
unique solution of
\begin{equation}\label{eq:mle}
A'(\hat\eta_n)=\frac1n\sum_{i=1}^n T(G(X_i)).
\end{equation}
It exists with probability tending to one, is strongly consistent, and
satisfies
\[
\sqrt n\,(\hat\eta_n-\eta)
\;\xrightarrow{d}\;
\mathcal N\!\bigl(0,\;1/A''(\eta)\bigr),
\]
where the Fisher information is
$I(\eta)=A''(\eta)=\Var_{p_\eta}(T(U))$. For the canonical family
($T(u)=u$), the Fisher information equals the Kolmogorov variance of the
member: $I(\eta)=V_G(f_\eta)$.
\end{proposition}

\begin{proof}
Strict concavity, existence, consistency, and asymptotic normality are the
classical theory of full natural exponential families
\cite{Brown1986,vanderVaart1998}, applied to the sample $U_1,\dots,U_n$
from $p_\eta$; the reduction to this classical setting is exactly
Proposition~\ref{prop:pushforward}. Equation~\eqref{eq:mle} is the score
equation $\partial_\eta\ell_n=S_n-nA'(\eta)=0$. The identification
$I(\eta)=V_G(f_\eta)$ for $T(u)=u$ is Theorem~\ref{thm:duality}(iii).
\end{proof}

Equation~\eqref{eq:mle} transports the classical moment-matching property
of exponential families. The empirical mean of
the sufficient statistic is matched to its population value, but the
matching takes place in probability coordinates, where the statistic is
bounded. By Theorem~\ref{thm:duality}, for $T(u)=u^r$ the estimator can
equivalently be described as matching the empirical $r$-th coordinate moment
$\widehat\mu_{r,n}=n^{-1}\sum_i(G(X_i))^r$ of \cite{Cojocea2026b} to
$A_r'(\eta)$; the plug-in Kolmogorov moment
estimators of \cite{Cojocea2026b} are maximum likelihood within the
corresponding $G$-exponential family.

\subsection{Robustness from geometric compactification}

\begin{proposition}[Influence function and gross-error sensitivity]
\label{prop:IF}
The maximum likelihood functional of the $G$-exponential family, viewed as
a functional on distributions via \eqref{eq:mle}, has influence function
\[
\IF(x;\hat\eta,f_\eta)
=\frac{T(G(x))-A'(\eta)}{A''(\eta)},
\qquad x\in I,
\]
which is bounded, with gross-error sensitivity
\[
\gamma^*(\hat\eta,f_\eta)
=\frac{\max\bigl(\tau_+-A'(\eta),\;A'(\eta)-\tau_-\bigr)}{A''(\eta)}
<\infty .
\]
For the canonical family, $\tau_-=0$ and $\tau_+=1$, so
$\gamma^*=\max\bigl(A'(\eta),\,1-A'(\eta)\bigr)/A''(\eta)$.
\end{proposition}

\begin{proof}
The functional is an M-functional with score
$\psi(x,\eta)=T(G(x))-A'(\eta)$; the influence function of an M-functional
is $\psi/(-\E[\partial_\eta\psi])=\psi/A''(\eta)$
\cite{HampelEtAl1986,Huber2009}. Boundedness follows from
$\tau_-\le T\le\tau_+$, and the supremum of $|\IF|$ is attained at the
essential bounds of $T$.
\end{proof}

\begin{remark}[Robustness by construction, with an honest caveat]
\label{rem:robustness}
In classical exponential families, the maximum likelihood estimator has
influence function proportional to $T(x)-A'(\eta)$ and is therefore
\emph{unbounded} whenever the sufficient statistic is unbounded---the
Gaussian mean being the canonical example. In $G$-exponential families the
sufficient statistic is bounded by construction, so maximum likelihood
itself is a bounded-influence procedure: no truncation, reweighting, or
Huberization is imposed, and full likelihood efficiency at the model is
retained. This is robustness from geometric compactification, the central
mechanism of \cite{Cojocea2026b}, here appearing not in a modified
estimator but in the model class itself.

The honest caveat, in the accounting of \cite{Cojocea2026b}, is that the
chart is fixed and known. The family \eqref{eq:gexp} is anchored at $G$:
it is not closed under location--scale changes of the data, and the
parameter $\eta$ measures position relative to the benchmark geometry, not
an equivariant location. When equivariance is required, the chart must be
calibrated from the data, and the two-step estimated-chart theory of
\cite{Cojocea2026b}, including its first-order calibration correction,
becomes the relevant framework; its extension from barycenters to full
$G$-exponential likelihoods is left for future work.
\end{remark}

\section{Entropy and information}\label{sec:entropy}

Two results complete the structural picture: a variational
characterization of the transported families, and the invariance of
information quantities under the transport.

For densities $f_1,f_2$ on the same space, write
$\KL(f_1\|f_2)=\int f_1\log(f_1/f_2)$ for the relative entropy
\cite{CoverThomas2006}.

\begin{proposition}[Minimum relative entropy to the chart law]
\label{prop:maxent}
Let $m$ lie in the interior of the range of $A'$, and let $\eta$ be the
unique solution of $A'(\eta)=m$. Then, among all probability densities $f$
on $I$ satisfying the probability-coordinate moment constraint
\[
\E_f\bigl[T(G(X))\bigr]=m,
\]
the member $f_\eta$ of the $G$-exponential family uniquely minimizes
$\KL(f\,\|\,g)$ over this class. Moreover, the Pythagorean identity
\[
\KL(f\,\|\,g)=\KL(f\,\|\,f_\eta)+\KL(f_\eta\,\|\,g)
\]
holds for every such $f$.
\end{proposition}

\begin{proof}
Let $f$ satisfy the constraint. If $\KL(f\|g)=\infty$ the inequality is
trivial, so assume it finite. Using
$\log(f/f_\eta)=\log(f/g)-\eta T(G(x))+A(\eta)$,
\[
\KL(f\,\|\,f_\eta)
=\KL(f\,\|\,g)-\eta\,\E_f[T(G(X))]+A(\eta)
=\KL(f\,\|\,g)-\eta m+A(\eta).
\]
For $f=f_\eta$, direct differentiation of $A$ gives
\[
\E_{f_\eta}[T(G(X))]=A'(\eta)=m.
\]
The same computation therefore gives
$\KL(f_\eta\|g)=\eta m-A(\eta)$. Adding the two displays yields the
Pythagorean identity, and since $\KL(f\|f_\eta)\ge0$ with equality iff
$f=f_\eta$ a.e., the minimization claim follows.
\end{proof}

\begin{remark}
Proposition~\ref{prop:maxent} is the transported analogue of the classical
maximum-entropy characterization of exponential families
\cite{CoverThomas2006}: the constraint is a moment constraint in
probability coordinates, and the reference law is the chart law. It
upgrades the informal statement that ``entropy maximization in probability
coordinates'' should generate the transported families---listed as an open
direction in earlier drafts of this circle of ideas---to a theorem.
\end{remark}

\begin{proposition}[Information invariance under transport]
\label{prop:invariance}
For all $\eta,\eta'\in N$,
\[
\KL\bigl(f_\eta\,\|\,f_{\eta'}\bigr)
=\KL\bigl(p_\eta\,\|\,p_{\eta'}\bigr),
\]
and the Fisher information functions of the two families coincide:
$I_f(\eta)=I_p(\eta)=A''(\eta)$.
\end{proposition}

\begin{proof}
In the ratio $f_\eta/f_{\eta'}=p_\eta(G(x))/p_{\eta'}(G(x))$ the chart
density cancels, so
\[
\KL(f_\eta\|f_{\eta'})
=\int_I p_\eta(G(x))\,g(x)\,
\log\frac{p_\eta(G(x))}{p_{\eta'}(G(x))}\,dx
=\int_0^1 p_\eta(u)\log\frac{p_\eta(u)}{p_{\eta'}(u)}\,du,
\]
by the substitution $u=G(x)$. The score
$\partial_\eta\log f_\eta(x)=T(G(x))-A'(\eta)$ is the score of $p_\eta$
composed with $G$, and its second moment under $f_\eta$ equals its second
moment under $p_\eta$ by Proposition~\ref{prop:pushforward}.
\end{proof}

\begin{remark}[What the chart does and does not deform]
\label{rem:invariance}
Proposition~\ref{prop:invariance} is an honest delimitation of the
construction, in the spirit of the degeneracy accounting of
\cite{Cojocea2026b}. The transport changes the geometry of value space
decisively---tails, support, barycenters, quantiles---but it changes
nothing about the statistical distinguishability structure of the family:
relative entropies, Fisher information, and hence the information geometry
in the sense of \cite{Amari2016} are those of the classical family on
$(0,1)$, transported isometrically. The framework does not manufacture new
information; it relocates classical exponential information into a
geometry where it coexists with heavy tails and bounded influence.
\end{remark}

\section{Worked example: the Cauchy chart}\label{sec:example}

Let $G$ be the standard Cauchy chart on $I=\R$,
\[
G(x)=\frac12+\frac1\pi\arctan x,
\qquad
g(x)=\frac1{\pi(1+x^2)},
\qquad
G^{-1}(u)=\tan\bigl(\pi(u-\tfrac12)\bigr),
\]
and consider the canonical family of Example~\ref{ex:canonical}, with
$T(u)=u$ and $A$, $A'$, $A''$ given by \eqref{eq:canonical-A}. The
transported densities are
\[
f_\eta(x)=\frac{\eta\,e^{\eta G(x)}}{e^\eta-1}\cdot\frac1{\pi(1+x^2)},
\qquad\eta\in\R,
\]
with $f_0=g$ the standard Cauchy density.
Figure~\ref{fig:cauchy-family} displays the family in both coordinate
systems: in probability coordinates the members are truncated exponential
densities on $(0,1)$; in value space they are asymmetric, heavy-tailed
deformations of the Cauchy law.

\begin{figure}[t]
\centering
\includegraphics[width=\textwidth]{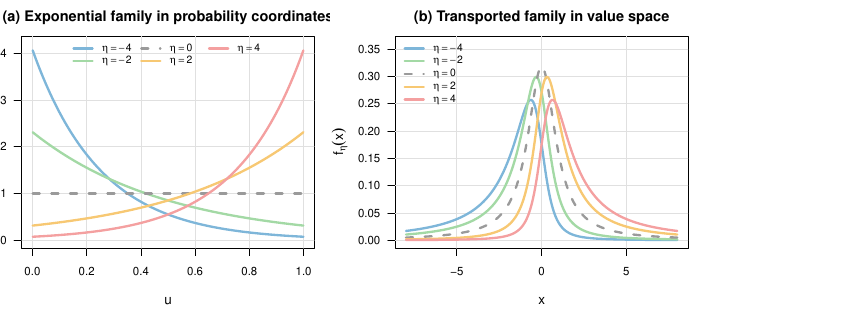}
\caption{The canonical $G$-exponential family for the standard Cauchy
chart, for $\eta\in\{-4,-2,0,2,4\}$. Panel (a): the exponential family
$p_\eta(u)=\eta e^{\eta u}/(e^\eta-1)$ in probability coordinates; the
dashed grey member is the uniform density ($\eta=0$). Panel (b): the
transported densities $f_\eta(x)=p_\eta(G(x))g(x)$ in value space; the
dashed grey member is the standard Cauchy density, and positive $\eta$
tilts probability mass towards the upper tail without altering the tail
index. All curves are exact.}
\label{fig:cauchy-family}
\end{figure}

\paragraph{Tails.}
Since $1-G(x)\sim1/(\pi x)$ and $G(x)\sim1/(\pi|x|)$ as $x\to+\infty$ and
$x\to-\infty$ respectively, and since
\[
c_\eta^+:=p_\eta(1^-)=\frac{\eta}{1-e^{-\eta}},
\qquad
c_\eta^-:=p_\eta(0^+)=\frac{\eta}{e^\eta-1},
\]
Theorem~\ref{thm:tail-equivalence} gives, as $x\to\infty$,
\[
1-F_\eta(x)\sim\frac{c_\eta^+}{\pi x},
\qquad
F_\eta(-x)\sim\frac{c_\eta^-}{\pi x}.
\]
Every member has exact Cauchy-type tails of index $1$
(Corollary~\ref{cor:rv}); the tilt redistributes tail weight between the
two tails through the ratio $c_\eta^+/c_\eta^-=e^\eta$ but cannot change
the decay. For $\eta=2$, for instance, $c_2^+\approx2.3130$ and
$c_2^-\approx0.3130$: the upper tail is $e^2\approx7.39$ times heavier
than the lower, yet both remain of Cauchy type.

\paragraph{Barycenters and medians.}
The probability barycenter of $f_\eta$ is
\[
b_G(f_\eta)=\tan\Bigl(\pi\bigl(A'(\eta)-\tfrac12\bigr)\Bigr),
\]
the pullback of the mean parameter in Theorem~\ref{thm:duality}(iii). The
median of $f_\eta$ is the pullback
$\operatorname{med}(f_\eta)=G^{-1}(u_\eta)$ of the median level
$u_\eta=\eta^{-1}\log\bigl((e^\eta+1)/2\bigr)$ of $p_\eta$. For $\eta=2$:
$A'(2)\approx0.6565$, $b_G(f_2)\approx0.536$, while $u_2\approx0.7169$
and $\operatorname{med}(f_2)\approx0.811$; for $\eta=4$:
$b_G(f_4)\approx1.125$ against $\operatorname{med}(f_4)\approx1.706$. The
Kolmogorov variance decreases from $A''(0)=1/12\approx0.0833$ to
$A''(2)\approx0.0690$ and $A''(4)\approx0.0435$ as the mass concentrates
near the boundary. Figure~\ref{fig:barycenter} traces both location
functionals across the family: the barycenter is systematically closer to
the anchor of the chart than the member median, a quantitative
illustration of the anchoring phenomenon analysed in
\cite{Cojocea2026b}---the fixed benchmark geometry reads locations
relative to its own centre, and shrinks them towards it.

\begin{figure}[t]
\centering
\includegraphics[width=0.72\textwidth]{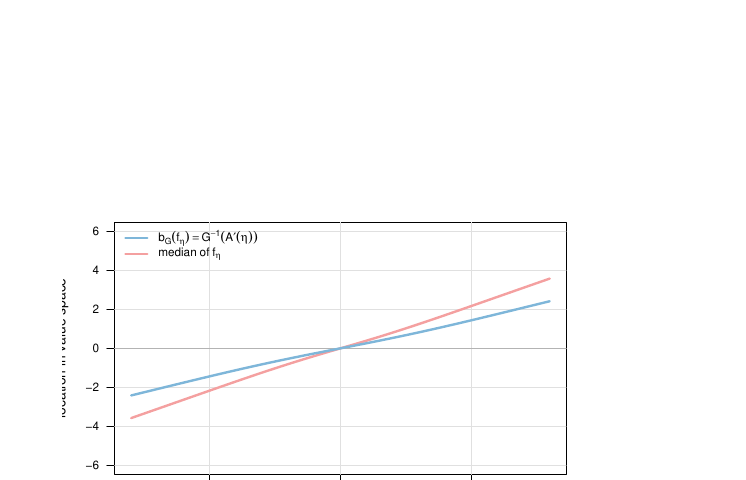}
\caption{Location functionals of the Cauchy-chart family as functions of
the natural parameter $\eta$. The probability barycenter
$b_G(f_\eta)=G^{-1}(A'(\eta))$ (blue) is the pullback of the mean
parameter of the exponential family in probability coordinates
(Theorem~\ref{thm:duality}); the member median (rose) is the pullback of
the median level of $p_\eta$. The barycenter is drawn towards the anchor
of the fixed chart, illustrating the non-equivariance analysed in the
companion inference paper. All curves are exact.}
\label{fig:barycenter}
\end{figure}

\paragraph{Inference.}
Given observations $X_1,\dots,X_n$ from $f_\eta$, the sufficient statistic
is $\bar S_n=\frac1n\sum_iG(X_i)=\frac1n\sum_i(\frac12+\frac1\pi\arctan
X_i)$, a bounded average, and the maximum likelihood estimator solves
$A'(\hat\eta_n)=\bar S_n$, a strictly monotone scalar equation. By
Proposition~\ref{prop:IF} with $A'(2)\approx0.6565$ and
$A''(2)\approx0.0690$, the gross-error sensitivity at $\eta=2$ is
$\gamma^*$ is finite, with
\[
\gamma^*=\frac{\max(0.6565,\,0.3435)}{0.0690}\approx9.52,
\]
in a model
whose every member has infinite mean. Maximum likelihood in this family is
a bounded-influence procedure on Cauchy-tailed data---classically a
setting where raw-moment inference is unavailable.

\section{Simulation study}\label{sec:simulation}

We corroborate the statistical theory of Section~\ref{sec:inference} with a
Monte Carlo study on the Cauchy-chart canonical family of
Section~\ref{sec:example}, a deliberately severe test case: every member
has infinite mean, so the ordinary sample mean cannot estimate a finite
location target consistently, yet the transported likelihood has bounded
influence by construction. All experiments use $B=2000$ independent
replications per configuration and the seed \texttt{2026}.

\paragraph{Design.}
Sampling from a member is exact, with no rejection step: by the pushforward
characterization (Proposition~\ref{prop:pushforward}), if
$V\sim\Unif(0,1)$ then $U=P_\eta^{-1}(V)$ has density $p_\eta$ and
$X=G^{-1}(U)$ has density $f_\eta$, where
$P_\eta^{-1}(v)=\eta^{-1}\log\bigl(1+v(e^\eta-1)\bigr)$ is the quantile
function of the truncated exponential $p_\eta$ and
$G^{-1}(u)=\tan(\pi(u-\tfrac12))$ is the Cauchy inverse chart. We compare
two feasible estimators of the natural parameter $\eta$: the maximum
likelihood estimator $\hat\eta_n$ solving the monotone score equation
$A'(\hat\eta_n)=\tfrac1n\sum_iG(X_i)$ of Proposition~\ref{prop:mle}, and a
robust competitor, the \emph{median-pullback estimator} $\tilde\eta_n$
solving $u_{\mathrm{med}}(\tilde\eta_n)=G(\operatorname{med}\{X_i\})$, where
$u_{\mathrm{med}}(\eta)=\eta^{-1}\log\bigl((e^\eta+1)/2\bigr)$ is the median
level $u_\eta$ of $p_\eta$ from Section~\ref{sec:example}; the estimator
reads $\eta$ off the sample median rather than the mean of the bounded
coordinates. Both are anchored at the fixed Cauchy chart, in the sense of
Remark~\ref{rem:robustness}.

\paragraph{Population efficiency.}
The maximum likelihood estimator attains the Fisher information
$I(\eta)=A''(\eta)$ (Proposition~\ref{prop:mle}), while the delta method
applied to the sample median gives, for the median-pullback estimator,
asymptotic variance
$1/\bigl(4\,p_\eta(u_{\mathrm{med}}(\eta))^2\,u_{\mathrm{med}}'(\eta)^2\bigr)$.
The resulting asymptotic relative efficiencies are
\[
\mathrm{ARE}(\tilde\eta:\hat\eta)=
\begin{cases}
0.750, & \eta=0,\\
0.671, & \eta=2,\\
0.562, & \eta=4,
\end{cases}
\]
so the likelihood estimator is the more efficient of the two throughout,
increasingly so as the tilt concentrates mass near the boundary of
probability space. At $\eta=0$ (the standard Cauchy law itself) the value
$\tfrac34$ is exact: $A''(0)=\tfrac1{12}$ is the variance of the uniform
coordinate, and $u_{\mathrm{med}}'(0)=\tfrac18$.

\paragraph{Consistency, efficiency, and coverage.}
Table~\ref{tab:consistency} reports bias, standard deviation, root mean
squared error, and the empirical coverage of the nominal $95\%$ Wald
interval $\hat\eta_n\pm z_{0.975}/\sqrt{n\,A''(\hat\eta_n)}$ built from the
plug-in Fisher information. Both estimators are essentially unbiased
already at $n=50$; the RMSE ratios track the population efficiencies above,
the maximum likelihood estimator dominating at every configuration; and the
Wald coverage sits at the nominal level across all values of $\eta$ and
$n$, confirming that the plug-in information of
Proposition~\ref{prop:mle} yields calibrated inference on data with no
finite mean. Figure~\ref{fig:sim-rmse} displays the same information in
scale-free form: $\sqrt n\times\text{RMSE}$ is flat in $n$ and settles on
the asymptotic standard deviations $\sqrt{A''(\eta)^{-1}}$ (dashed), the
signature of $\sqrt n$-consistency at the efficient rate.

\begin{table}[t]
\centering
\small
\begin{tabular}{cc rrrc rrr}
\toprule
& & \multicolumn{4}{c}{MLE $\hat\eta_n$}
  & \multicolumn{3}{c}{median-pullback $\tilde\eta_n$}\\
\cmidrule(lr){3-6}\cmidrule(lr){7-9}
$\eta$ & $n$ & bias & SD & RMSE & cover & bias & SD & RMSE\\
\midrule
% --- Monte Carlo values (B = 2000, seed 2026) ---
$0$ & 50 & 0.0140 & 0.4970 & 0.4970 & 0.952 & 0.0114 & 0.5755 & 0.5755 \\
 & 200 & -0.0029 & 0.2483 & 0.2482 & 0.945 & -0.0064 & 0.2825 & 0.2825 \\
 & 1000 & 0.0012 & 0.1127 & 0.1126 & 0.946 & -0.0008 & 0.1273 & 0.1273 \\
 & 5000 & -0.0027 & 0.0476 & 0.0476 & 0.956 & -0.0034 & 0.0547 & 0.0548 \\
\addlinespace
$2$ & 50 & 0.0269 & 0.5498 & 0.5504 & 0.953 & 0.0518 & 0.6810 & 0.6828 \\
 & 200 & 0.0038 & 0.2740 & 0.2739 & 0.947 & 0.0025 & 0.3311 & 0.3311 \\
 & 1000 & -0.0013 & 0.1180 & 0.1180 & 0.954 & -0.0039 & 0.1433 & 0.1433 \\
 & 5000 & 0.0004 & 0.0551 & 0.0551 & 0.945 & 0.0003 & 0.0677 & 0.0677 \\
\addlinespace
$4$ & 50 & 0.0536 & 0.7008 & 0.7027 & 0.952 & 0.1015 & 0.9554 & 0.9605 \\
 & 200 & 0.0197 & 0.3469 & 0.3474 & 0.951 & 0.0260 & 0.4568 & 0.4575 \\
 & 1000 & 0.0056 & 0.1511 & 0.1512 & 0.948 & 0.0107 & 0.2011 & 0.2013 \\
 & 5000 & -0.0003 & 0.0679 & 0.0679 & 0.951 & 0.0009 & 0.0894 & 0.0894 \\
\bottomrule
\end{tabular}
\caption{Finite-sample behaviour of the two estimators of the natural
parameter $\eta$ for the Cauchy-chart canonical family
($B=2000$ replications, seed \texttt{2026}). ``cover'' is the empirical
coverage of the nominal $95\%$ Wald interval based on the plug-in Fisher
information $A''(\hat\eta_n)$. Every member of this family has infinite
mean, so the sample mean of the raw data is not a consistent estimator of
any location parameter; the estimators here operate on the bounded
coordinates $G(X_i)$.}
\label{tab:consistency}
\end{table}

\begin{figure}[t]
\centering
\makebox[\textwidth]{%
  \hfill (a) $\eta=0$\hfill
  (b) $\eta=2$\hfill
  (c) $\eta=4$\hfill}
\vspace{-0.5em}
\includegraphics[width=\textwidth]{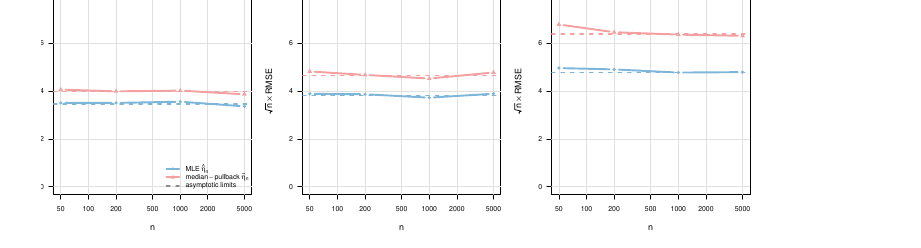}
\caption{Scale-free root mean squared error $\sqrt n\times\text{RMSE}$ of
the maximum likelihood estimator (blue) and the median-pullback estimator
(rose) as a function of $n$ (log axis), for $\eta\in\{0,2,4\}$; $B=2000$
replications per point. Dashed lines mark the asymptotic standard
deviations $\sqrt{A''(\eta)^{-1}}$ and
$\bigl(4p_\eta(u_{\mathrm{med}})^2u_{\mathrm{med}}'^2\bigr)^{-1/2}$. The
maximum likelihood estimator is uniformly the more efficient, and both
curves are flat, confirming $\sqrt n$-consistency.}
\label{fig:sim-rmse}
\end{figure}

\paragraph{Asymptotic normality.}
Figure~\ref{fig:sim-normality} overlays the standard normal density on the
histogram of the studentized statistic
$\sqrt{n\,A''(\hat\eta_n)}\,(\hat\eta_n-\eta)$ at $\eta=2$, for $n=50$ and
$n=1000$. The match is already good at $n=50$ (Monte Carlo mean $0.003$,
standard deviation $0.987$) and excellent at $n=1000$ (mean $-0.022$,
standard deviation $1.018$), illustrating
Proposition~\ref{prop:mle} on Cauchy-tailed data. The studentization uses
only the bounded coordinates, so nothing in the procedure is sensitive to
the tails of the raw sample.

\begin{figure}[t]
\centering
\includegraphics[width=\textwidth]{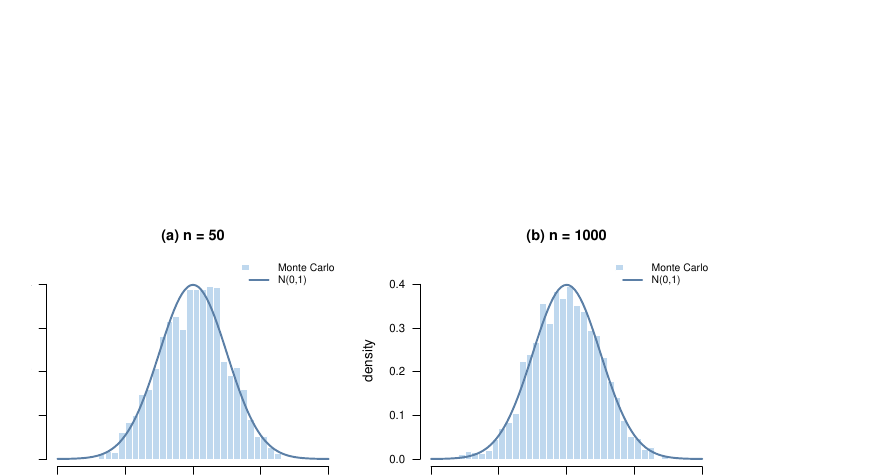}
\par\vspace{-0.6em}
{\small $\sqrt{nA''(\hat\eta_n)}\,(\hat\eta_n-\eta)$}
\par
\caption{Sampling distribution of the studentized maximum likelihood
statistic $\sqrt{n\,A''(\hat\eta_n)}\,(\hat\eta_n-\eta)$ at $\eta=2$ for the
Cauchy-chart family, with the $N(0,1)$ density overlaid; $B=2000$
replications. The Gaussian approximation is accurate already at $n=50$
(panel a) and essentially exact at $n=1000$ (panel b).}
\label{fig:sim-normality}
\end{figure}

\paragraph{Robustness under contamination.}
The final experiment places a fraction $\varepsilon$ of the mass of an
$\eta=2$ sample ($n=500$) at the fixed outlier $x_0=50$ and records the
induced bias of each estimator. Table~\ref{tab:robustness} compares the
Monte Carlo means with the \emph{exact} contaminated functionals: for the
maximum likelihood estimator the contaminated target solves
$A'(\eta_\varepsilon)=(1-\varepsilon)A'(2)+\varepsilon\,G(x_0)$, a direct
consequence of the estimating equation, and the median-pullback target is
obtained analogously from the contaminated median. The agreement is within
Monte Carlo error at every level, and the exact influence-function slope
$\IF(x_0)=\bigl(G(x_0)-A'(2)\bigr)/A''(2)\approx4.89$ predicts the initial
bias. Crucially the response is \emph{bounded}: as $x_0\to\infty$ the
outlier's leverage saturates at $\bigl(1-A'(2)\bigr)/A''(2)\approx4.98$ per
unit of $\varepsilon$, because the coordinate $G(x_0)$ cannot exceed $1$
(Figure~\ref{fig:sim-contamination}). A single arbitrarily large
\emph{positive} observation therefore moves the estimate by a finite,
computable amount---at most the upper-tail saturation
$\bigl(1-A'(2)\bigr)/A''(2)\approx4.98$ per unit of $\varepsilon$---the
exact opposite of what an unbounded sufficient statistic would produce. The
two-sided gross-error sensitivity $\gamma^*\approx9.52$ of
Proposition~\ref{prop:IF} is the larger, opposite-tail bound $A'(2)/A''(2)$,
attained only in the limit $x_0\to-\infty$ (a large \emph{negative} outlier,
where $G(x_0)\to0$), a configuration this experiment does not exercise. Over
most of the contamination range the maximum likelihood estimator is also the
more robust of the two, its exact bias curve lying below that of the
median-pullback for $\varepsilon\gtrsim0.05$; the two exact curves cross near
$\varepsilon\approx0.047$, below which the median-pullback carries a
marginally smaller bias---its initial slope $4.65$ against the maximum
likelihood estimator's $\IF(x_0)\approx4.89$---a gap that is within Monte
Carlo error.

\begin{table}[t]
\centering
\small
\begin{tabular}{c rr@{\ }l rr@{\ }l}
\toprule
& \multicolumn{3}{c}{MLE $\hat\eta_n$}
  & \multicolumn{3}{c}{median-pullback $\tilde\eta_n$}\\
\cmidrule(lr){2-4}\cmidrule(lr){5-7}
$\varepsilon$ & exact & \multicolumn{2}{c}{Monte Carlo}
              & exact & \multicolumn{2}{c}{Monte Carlo}\\
\midrule
% --- Monte Carlo values (B = 2000, seed 2026, n = 500, x0 = 50) ---
0.00 & 0.0000 & 0.0013 & (0.0038) & 0.0000 & 0.0036 & (0.0047) \\
0.02 & 0.0986 & 0.0976 & (0.0039) & 0.0957 & 0.1027 & (0.0049) \\
0.05 & 0.2500 & 0.2469 & (0.0040) & 0.2509 & 0.2483 & (0.0051) \\
0.10 & 0.5129 & 0.5139 & (0.0044) & 0.5477 & 0.5535 & (0.0059) \\
0.15 & 0.7911 & 0.8039 & (0.0046) & 0.9084 & 0.9320 & (0.0066) \\
0.20 & 1.0881 & 1.0859 & (0.0050) & 1.3634 & 1.3713 & (0.0087) \\
\bottomrule
\end{tabular}
\caption{Bias of the two estimators under $\varepsilon$-contamination at
$x_0=50$ ($\eta=2$, $n=500$, $B=2000$ replications, seed \texttt{2026}).
``exact'' is the bias of the population contaminated functional; the Monte
Carlo mean is reported with its standard error in parentheses. The two
agree within simulation error, and the bias remains finite and bounded at
every contamination level.}
\label{tab:robustness}
\end{table}

\begin{figure}[t]
\centering
\includegraphics[width=0.78\textwidth]{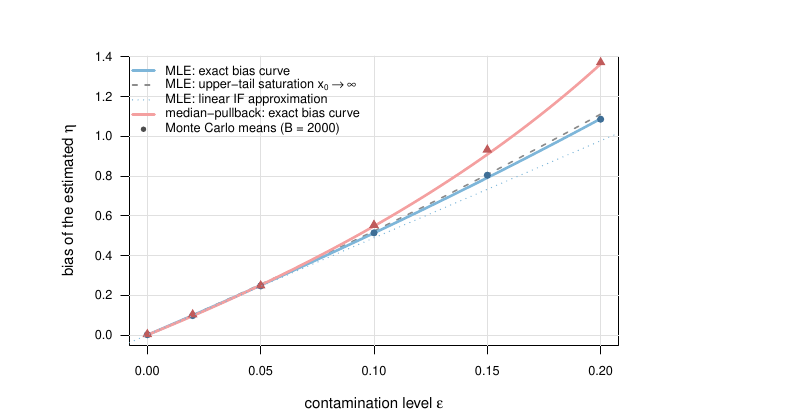}
\caption{Bias under $\varepsilon$-contamination at $x_0=50$ for the
Cauchy-chart family ($\eta=2$, $n=500$). Solid curves are the exact
contaminated functionals of the maximum likelihood estimator (blue) and
the median-pullback estimator (rose); markers are Monte Carlo means over
$B=2000$ replications. The dotted line is the linear influence-function
approximation $\varepsilon\,\IF(x_0)$, and the dashed line is the
upper-tail saturation slope as $x_0\to\infty$, at which a positive outlier's
leverage saturates because $G(x_0)<1$. Robustness is a structural consequence of
the bounded coordinate, not of any explicit downweighting.}
\label{fig:sim-contamination}
\end{figure}

\paragraph{Summary.}
The study confirms, on a family with no finite mean, the three statistical
claims of Section~\ref{sec:inference}: the transported maximum likelihood
estimator is $\sqrt n$-consistent and asymptotically efficient with
calibrated Wald inference, its studentized law is Gaussian in moderate
samples, and its influence is bounded with the exact gross-error
sensitivity predicted by the theory. Robustness here is not imposed but
inherited from the compactness of probability coordinates, exactly the
mechanism of the companion inference paper \cite{Cojocea2026b}.

\section{Comparison with deformed exponentials}
\label{sec:deformed}

Generalized exponential families in the literature are predominantly
obtained by algebraic deformation. The Tsallis $q$-exponential
\cite{Tsallis1988},
\[
\exp_q(t)=\bigl[1+(1-q)t\bigr]_+^{1/(1-q)},
\]
the Kaniadakis $\kappa$-exponential \cite{Kaniadakis2002}, and the
deformed-logarithm families systematized by Naudts \cite{Naudts2011}
modify the exponential function so that log-affine structure coexists with
power-law decay; R\'enyi's construction \cite{Renyi1961} instead deforms
the aggregation of probability weights inside entropy, using generalized
means in the Kolmogorov--Nagumo lineage \cite{Kolmogorov1930,Nagumo1930}.
Escort transformations \cite{Tsallis1988} act at yet another level,
reweighting the probability measure itself.

The present construction touches none of these levers: the exponential
function is classical, the aggregation is linear, and the measure is
untouched. What is deformed is the coordinate system through which
exponential structure is read, and Sections
\ref{sec:tails}--\ref{sec:entropy} quantify exactly what this geometric
lever does (tails, support, barycenters, robustness) and does not
(relative entropy, Fisher information) control.

The two routes can represent the same law, and the comparison is
instructive. The standard Cauchy density admits the algebraic
representation
\[
\frac1{\pi(1+x^2)}=\frac1\pi\exp_2(-x^2),
\]
a $q$-exponential quadratic-energy model with $q=2$. In the present
framework the same law appears instead as the \emph{chart law}---the
$\eta=0$ member of the Cauchy-chart family of
Section~\ref{sec:example}---and its heavy tails are carried by the
geometry $G$, which then transmits them uniformly to every member of the
family (Corollary~\ref{cor:rv}). In the algebraic reading, heavy tails are
produced by weakening the exponential function; in the geometric reading,
they are produced by the choice of probability coordinates, while the
exponential function operates intact on the compact probability scale. The
statistical consequences differ accordingly: bounded sufficient statistics
and bounded-influence likelihood are automatic in the geometric reading
(Section~\ref{sec:inference}), because boundedness is supplied by the
coordinate system rather than negotiated with the tail of the deformed
exponential.

We deliberately do not use the notation $\exp_G$ for any deformed
function: in this framework no deformed exponential exists, and attaching
the chart symbol to an algebraically deformed function would misstate
where the deformation lives.

\section{Towards intrinsic geometric exponentiality}\label{sec:intrinsic}

The transported theory developed above is the rigorous layer of a broader
question. Transported exponentiality still applies the ordinary
exponential function in probability coordinates; one may ask whether
probability geometry admits \emph{intrinsically} geometric notions of
logarithm and exponential---operations $\log_G$ and $\exp_G$ genuinely
constructed from the chart, for which log-affine structure in the
resulting scale defines a strictly larger class of models than
\eqref{eq:gexp}.

We record the question but do not pursue it here, for a reason the
companion papers make familiar: the constraints such operations must
satisfy---algebraic consistency, compatibility with normalization,
invariance under the reparametrizations of \cite{Cojocea2026a}, and a
statistical interpretation of the resulting sufficient statistics---are
severe, and nothing in the transported theory forces a canonical choice.
The rigidity phenomena of \cite{Cojocea2026a} (only affine
reparametrizations of the probability scale preserve barycenters) suggest
that a canonical intrinsic exponential, if one exists, will be tightly
constrained. Until such a construction is exhibited and its admissibility
established, intrinsic geometric exponentiality remains a programmatic
direction, and the claims of this paper are confined to the transported
framework, where every statement is a theorem.

\section{Discussion and outlook}\label{sec:discussion}

This paper extended the probability-coordinate framework of the companion
papers from location functionals to parametric families: classical
exponential structure, imposed on the compact probability scale and
transported through a chart, generates families in which heavy tails,
asymmetry, and boundary concentration are geometric inheritances from the
chart rather than algebraic properties of a deformed exponential.

Three structural findings organize the theory. First, the pushforward
identity: a $G$-exponential family is precisely the family whose
probability coordinates follow a classical exponential family, so the
entire barycentric machinery of \cite{Cojocea2026a} computes its canonical
parameters---the probability barycenter is the pulled-back mean parameter,
and the Kolmogorov variance is the Fisher information. Second, the
separation of levers: the chart controls tails, support, and value-space
geometry (Theorem~\ref{thm:tail-equivalence}), while the exponential
structure controls distinguishability, and the two do not interact
(Proposition~\ref{prop:invariance})---tilting is tail-neutral, and
transport is information-neutral. Third, robustness by construction: the
sufficient statistic is bounded because the coordinate system is bounded,
so maximum likelihood has bounded influence with no modification of the
likelihood, realizing robustness through geometric compactification as in
\cite{Cojocea2026b} at the level of model classes.

Several directions remain open. The anchoring caveat of
Section~\ref{sec:inference} calls for an estimated-chart extension: a
two-step theory in the sense of \cite{Cojocea2026b} in which the chart
family $G_{\mu,\sigma}=G_0((\cdot-\mu)/\sigma)$ is calibrated before the
exponential parameter is estimated, restoring equivariance at the price of
a calibration correction. Chart selection---which benchmark geometry to
impose, and how to compare fits across charts given the information
invariance of Proposition~\ref{prop:invariance}---is a model-choice
problem specific to this framework. Multivariate extensions through the
copula coordinates of \cite{Cojocea2026a} would transport exponential
families on the unit cube. Finally, the intrinsic question of
Section~\ref{sec:intrinsic}---whether probability geometry generates its
own exponential calculus---remains, in our view, the deepest open
direction, and the transported theory developed here is intended as its
rigorous point of departure.

\end{document}